\documentclass[a4paper,fleqn]{cas-sc}
\usepackage[authoryear,longnamesfirst]{natbib}
\usepackage{algorithm}
\usepackage{algpseudocode}
\usepackage{amsmath,amssymb}
\usepackage{amsthm}
\usepackage{subcaption}
\usepackage{caption}
\newtheorem{assumption}{Assumption}
\newtheorem{lemma}{Lemma}
\newtheorem{theorem}{Theorem}

\newtheorem{example}{Example}
\newtheorem{definition}{Definition}
\def\tsc#1{\csdef{#1}{\textsc{\lowercase{#1}}\xspace}}
\tsc{WGM}
\tsc{QE}

\usepackage{enumitem}
\usepackage{titlesec}
\titleformat{\section}  {\normalfont\bfseries\fontsize{12pt}{14pt}\selectfont}
  {\thesection.}{0.8em}{}
\titlespacing*{\section}{0pt}{10pt plus 2pt minus 2pt}{8pt plus 2pt minus 2pt}
\DeclareMathSizes{10}{8.5}{6.5}{5} 
\usepackage{setspace}
\begin{document}
\let\WriteBookmarks\relax
\def\floatpagepagefraction{1}
\def\textpagefraction{.001}
\shorttitle{}    
\shortauthors{}  
\title [mode = title]{A greedy double-block nonlinear Bregman-Kaczmarz method for nonlinear equations }  
\tnotemark[1] 
\tnotetext[1]{} 
\author[ ]{Wendi Bao}
\cormark[1]
\fnmark[1]
\ead{baowd@upc.edu.cn}
\credit{}
\author[ ]{Jing Li}
\fnmark[2]
\author[ ]{Lili Xing}
\author[ ]{Weiguo Li}
\credit{}
\affiliation[]{organization={College of Science},
            addressline={ China University of Petroleum}, 
            city={Qingdao},
            postcode={266580}, 
            state={P.R.},
            country={China}}
\cortext[1]{Corresponding author}
\fntext[1]{}
\begin{keywords}
Nonlinear equations, column-blocked method, Gauss-Seidel method,  Bregman distance
\end{keywords}
\maketitle
\begin{abstract}
In this paper, we develop a  greedy double-block nonlinear Bregman-Kaczmarz method to efficiently find sparse solutions of nonlinear equations. The theoretical convergence rate of the proposed method is established under certain assumptions. Numerical experiments illustrate the competitive performance and superiority of the proposed method.
\end{abstract}
\section{Introduction}\label{sec1}

Sparse solution recovery for the constrained nonlinear optimization problem
\begin{equation}\label{eq:1.2}
\mathop{\min}_{x\in \mathcal{D}}\varphi(x),\quad \mathrm{s.t.} \quad F(x)=0,
\end{equation}
where $\varphi:\mathbb{R}^n\to \overline{\mathbb{R}}$ is a convex, nonsmooth sparsity-promoting function with nonempty domain $\mathrm{dom}\,\varphi=\{x\in \mathbb{R}^n:\varphi(x)<\infty\}$, which arises in many practical applications;
see \cite{Wang21,Zhang23,Lv24,Ye24,Bao25,Jiang25}.
By integrating the Kaczmarz iteration into the Bregman scheme, the sparse Kaczmarz method was established in \cite{Lorenz14_1,Lorenz14_2}.
The Bregman method \cite{Yin08} was originally developed to solve the regularized basis pursuit problem
\begin{equation}
\mathop{\min}_{x\in\mathbb{R}^n}\lambda\|x\|_1+\frac{1}{2}\|x\|_2^2\quad \mathrm{s.t.}\quad Ax=b,
\end{equation}
which provides an effective framework for solving the sparse solutions of linear systems.  Based on this framework, numerous improved variants have been proposed to accelerate convergence. 

However, research on sparse solution recovery for nonlinear systems remains relatively underdeveloped. In this context, Gower \cite{Gower24} proposed the nonlinear Bregman-Kaczmarz (NBK) and relaxed nonlinear Bregman-Kaczmarz (rNBK) methods by embedding Bregman projections into the nonlinear Kaczmarz (NK) framework. To accelerate convergence, Xiao et al. \cite{Xiao25} developed the averaged fast block nonlinear Bregman-Kaczmarz with adaptive stepsize (ABNBK‑a) method for sparse nonlinear signal recovery, which is a row-block variant of NBK. Inspired by the work \cite{Xiao25}, we propose the greedy double-block nonlinear Bregman-Kaczmarz (GDBNK) method for solving sparse solutions to nonlinear equations.


Notation used throughout this paper is summarized below. For any matrix $A\in \mathbb{R}^{m\times n}$, $\| \cdot\|_2$, $\sigma_{\min}$, $\sigma_{\max}$ denote the Euclidean norm, the minimum and maximum nonzero singular values of a matrix $A$, respectively. Let $\nabla F(x)$ be the Jacobian matrix of $F(x)$, 
 $\widehat{x}$ be a solution of the nonlinear system
$F(x)=0$ such that
$F(\widehat{x})=0$.
For an index set
$I_k\subseteq\{1,2,\ldots,m\}$,
we denote by $F_{I_k}(x)$ the subvector of $F(x)$ indexed by $I_k$.
For an index set
$J_k\subseteq\{1,2,\ldots,n\}$,
$\nabla F_{I_k,J_k}(x)$ denotes the submatrix of the Jacobian matrix
$\nabla F(x)$ whose row and column indices are determined by
$I_k$ and $J_k$, respectively.
Let
$I_{:,J_k}\in\mathbb{R}^{n\times |J_k|}$
denote the column-selection matrix consisting of the columns of the identity matrix $I_n$ indexed by $J_k$,  then we have
   $\nabla F_{I_k,J_k}(x) =
    \nabla F_{I_k,:}(x)I_{:,J_k}.$   Let $\mathbb{E}_k[\cdot]$ denote the conditional expectation with
respect to the random choice of $J_k$, conditioned on all
information available before selecting $J_k$ at the $k$th
iteration.

The rest of this paper is organized as follows.  We propose the GDBNK method in Section \ref{sec2} and conduct its convergence analysis in Section \ref{sec3}. Numerical experiments are carried out in Section \ref{sec4} to validate the performance of the proposed method. Finally,  we
draw the conclusions in Section \ref{sec5}.

\section{Greedy double-block nonlinear Bregman-Kaczmarz method}\label{sec2}
First, we present key assumption, difinition  and lemmas required throughout this paper.



\begin{definition}
\label{def:tcc}
\cite{Kaltenbacher08} Let $F(x)$ be a continuously differentiable mapping. Then $F(x)$ satisfies the
tangential cone condition in $\Omega\subset \mathbb{R}^n$ if there exists $0<\eta<1$ such that for any $x_1,x_2\in\Omega$,
\setlength{\abovedisplayskip}{4pt plus 2pt minus 2pt}
\setlength{\belowdisplayskip}{4pt plus 2pt minus 2pt}
\begin{equation} \label{def422}
|F(x_1)-F(x_2)-\nabla F(x_1)(x_1-x_2)|\leq\eta|F(x_1)-F(x_2)|.
\end{equation}
\end{definition}
	
	\begin{assumption}\label{asm1}
The functions $F(x)$ and $\varphi(x)$ in \eqref{eq:1.2} satisfy the following assumptions.
\begin{itemize}[noitemsep,topsep=0pt,parsep=0pt,partopsep=0pt]	
	\item[(i)] The nonlinear function $F(x):\mathcal{D}\subseteq\mathbb{R}^n\to\mathbb{R}^m$ on a bounded closed $\mathcal{D}$ is derivable and the derivative of $F(x)$ is continuous in $\mathcal{D}$.
	\item[(ii)] $F(x)$ satisfies the tangential cone condition.
	\item[(iii)] $\varphi(x):\mathbb{R}^n\to \mathbb{R}$ is proper, convex and lower semicontinuous.
	\item[(iv)] $\nabla F(\hat{x})$ is full column rank matrix.	
\end{itemize}
\end{assumption}
\begin{lemma}
\label{lem:1+eta^2}
\cite{Zhang24} If the nonlinear function $F(x)$ satisfies the local tangential cone condition in $\mathcal{D}$, then for all $x_1, x_2 \in \mathcal{D}$ and an index subset $\tau\subseteq [m]$, we have
\setlength{\abovedisplayskip}{4pt plus 2pt minus 2pt}
\setlength{\belowdisplayskip}{4pt plus 2pt minus 2pt}
\begin{equation}\label{eq:block ltcc}
\frac{1}{(1+\eta)^2}\|F_{\tau,:}^\prime(x_1)(x_1-x_2)\|_2^2\leq \|F_{\tau,:}(x_1)-F_{\tau,:}(x_2)\|_2^2\leq\frac{1}{(1-\eta)^2}\|F_{\tau,:}^\prime(x_1)(x_1-x_2)\|_2^2.
\end{equation}
\end{lemma}
		
		
		

\begin{lemma}\label{lem5}
Assume that Assumption \ref{asm1} (i) holds, 
then  there exist constants
$\underline{\nu} $ and $\overline{\nu}$
such that
\setlength{\abovedisplayskip}{4pt plus 2pt minus 2pt}
\setlength{\belowdisplayskip}{4pt plus 2pt minus 2pt}
\begin{equation}
    \inf_{x\in \mathcal{D}}\sigma_{\min}
\left(
\nabla F_{I_k,J_k}(x)
\right)
=
\underline{\nu}>0,
\qquad
 \sup_{x\in \mathcal{D}}\sigma_{\max}
\left(
\nabla F_{I_k,J_k}(x)
\right)
=
\overline{\nu}<\infty.
\end{equation}
\end{lemma}


To accelerate convergence and improve computational efficiency, we update only a subset of Jacobian columns instead of the full column set. Accordingly, we develop the greedy double-block nonlinear Bregman-Kaczmarz (GDBNK)  method to recover sparse solutions to nonlinear equations, as outlined in Algorithm \ref{alg3}.

\begin{algorithm}[htbp]
\caption{\textbf{ A greedy double-block nonlinear Bregman-Kaczmarz (GDBNK)  method }}\label{alg3}
\begin{algorithmic}[1]
\Require $x_0^* \in \mathbb{R}^n$, $x_0 = \nabla\varphi^*(x_0^*)$, $\gamma>0$, $\delta>0$, $\theta\in(0,1]$, $1<\tau < n$, $l>0$
\For{$k=0,1,\dots,l-1$}
    \State Select greedy row set $I_k=\bigl\{i_k \mid |F^{(i_k)}(x_k)|^2 \ge \theta\cdot\max_{1\le i \le m}|F^{(i)}(x_k)|^2\bigr\}$
    \State Randomly select column index set $J_k\subset\{1,2,\dots,n\}$ uniformly at random with $|J_k|=\tau$
    \State Update $x_{k+1}^*=x_{k}^*-\delta\frac{\gamma\|F_{I_k}(x_k)\|_2^2}{\|\nabla F_{I_kJ_k}^T(x_k)F_{I_k}(x_k)\|_2^2}I_{:,J_k}\nabla F_{I_kJ_k}^T(x_k)F_{I_k}(x_k)$
    \State Update $x_{k+1}=\nabla\varphi^{*}(x_{k+1}^*)$
    \State Compute residual: $r_{k+1}=-F(x_{k+1})$
\EndFor
\State \Return $x_l$
\end{algorithmic}
\end{algorithm}
\section{Convergence analysis}\label{sec3}
\begin{theorem}
\label{thm:main-convergence}
Suppose that Assumption \ref{asm1} holds
and $\varphi(x)$ is $\gamma$-strongly convex and $M$-smooth with respect to $\|\cdot\|_2$. Let  $\underline{\nu}$ and $\overline{\nu}$ be defined as in Lemma \ref{lem5}.  The sequence $\{x_k\}$ and $\{x_k^*\}$ are generated by Algorithm \ref{alg3}.
Assume that
\begin{equation}\label{23}
  \frac{1-\eta}{\overline{\nu}^{\,2}}
-
\left(
\frac{1}{\underline{\nu}^{\,2}}
-
\frac{1}{\overline{\nu}^{\,2}}
\right)
\frac{
\overline{\nu}(1+\eta)\sqrt{m}
}{
\sqrt{\theta}\,
\sigma_{\min}
\bigl(
\nabla F(\widehat{x})
\bigr)
}
>0,  
\end{equation}
if parameter $\delta$ satisfies
\begin{align}
\label{24}
0<\delta
<
2\underline{\nu}^{\,2}\frac{\tau}{n}
\Bigg[
&
\frac{1-\eta}{\overline{\nu}^{\,2}}
-\left(
\frac{1}{\underline{\nu}^{\,2}}
-
\frac{1}{\overline{\nu}^{\,2}}
\right)
\frac{
\overline{\nu}(1+\eta)\sqrt{m}
}{
\sqrt{\theta}\,
\sigma_{\min}
\bigl(
\nabla F(\widehat{x})
\bigr)
}
\Bigg],
\end{align}
then there exists a constant $c>0$ such that
\begin{equation}
\label{eq:linear-bregman-convergence}
\mathbb{E}
\left[
D_{\varphi}^{x_k^{*}}
(x_k,\widehat{x})
\right]
\leq
(1-c)^k
D_{\varphi}^{x_0^{*}}
(x_0,\widehat{x}),
\end{equation}
where
\begin{equation}\label{256}
    c
=
\frac{
\gamma\delta\theta
\sigma_{\min}^{2}
\bigl(
\nabla F(\widehat{x})
\bigr)
}{
mM(1+\eta)^2
}
\times
\Bigg[
\frac{\tau}{n}
\frac{1-\eta}{\overline{\nu}^{\,2}}
-
\frac{\tau}{n}
\left(
\frac{1}{\underline{\nu}^{\,2}}
-
\frac{1}{\overline{\nu}^{\,2}}
\right)
\frac{
\overline{\nu}(1+\eta)\sqrt{m}
}{
\sqrt{\theta}\,
\sigma_{\min}
\bigl(
\nabla F(\widehat{x})
\bigr)
}
-
\frac{\delta}{2\underline{\nu}^{\,2}}
\Bigg].
\end{equation}
In particular, if $0<c<1$, then the sequence generated by
Algorithm \ref{alg3} converges linearly to $\widehat{x}$. 
Furthermore,  if $x_0^*=x_0=0$,  the error between  $x_k$  and the solution $\hat{x}$ satisfies: 
\begin{equation}\label{eq111}
\mathbb{E}
\left[
\left\|
x_k-\widehat{x}
\right\|_2^{2}
\right]
\leq
\frac{2}{\gamma}
(1-c)^k
D_{\varphi}^{x_0^{*}}
(x_0,\widehat{x}).
\end{equation}
\end{theorem}
\begin{proof}
Since $\varphi$ is $\gamma$-strongly convex and
$x_k^{*}\in\partial\varphi(x_k)$,
according to Algorithm \ref{alg3}, we  can get that
\begin{align}
\label{eq8}
D_{\varphi}^{x_{k+1}^{*}}
(x_{k+1},\widehat{x})
&
\leq \nonumber
D_{\varphi}^{x_k^{*}}
(x_k,\widehat{x})
+
\left\langle
x_{k+1}^{*}-x_k^{*},
x_k-\widehat{x}
\right\rangle
+
\frac{1}{2\gamma}
\left\|
x_{k+1}^{*}-x_k^{*}
\right\|_2^{2}.\\
&\leq \nonumber
D_{\varphi}^{x_k^{*}}
(x_k,\widehat{x})-
\delta\gamma
\frac{
\left\|F_{I_k}(x_k)\right\|_2^{2}
}{
\left\|
\nabla F_{I_k,J_k}^{T}(x_k)
F_{I_k}(x_k)
\right\|_2^{2}
}
F_{I_k}(x_k)^{T}
\nabla F_{I_k,J_k}(x_k)
I_{:,J_k}^{T}
(x_k-\widehat{x})\\
&+
\frac{\delta^{2}\gamma}{2}
\frac{
\left\|F_{I_k}(x_k)\right\|_2^{4}
}{
\left\|
\nabla F_{I_k,J_k}^{T}(x_k)
F_{I_k}(x_k)
\right\|_2^{2}
}.
\end{align}
 Taking the conditional expectation in
\eqref{eq8}, we obtain
\begin{align}
\mathbb{E}_k
\left[
D_{\varphi}^{x_{k+1}^{*}}
(x_{k+1},\widehat{x})
\right]
\nonumber
&\leq
D_{\varphi}^{x_k^{*}}
(x_k,\widehat{x})
\nonumber-
\frac{\delta\gamma}{\binom{n}{\tau}}
\sum_{|J_k|=\tau}
\frac{
\left\|F_{I_k}(x_k)\right\|_2^{2}
}{
\left\|
\nabla F_{I_k,J_k}^{T}(x_k)
F_{I_k}(x_k)
\right\|_2^{2}
}\nonumber
F_{I_k}(x_k)^{T}
\nabla F_{I_k,J_k}(x_k)
I_{:,J_k}^{T}
(x_k-\widehat{x})
\nonumber\\
&+
\frac{\delta^{2}\gamma}
{2\binom{n}{\tau}}
\sum_{|J_k|=\tau}
\frac{
\left\|F_{I_k}(x_k)\right\|_2^{4}
}{
\left\|
\nabla F_{I_k,J_k}^{T}(x_k)
F_{I_k}(x_k)
\right\|_2^{2}
}.
\label{eq:conditional-expectation}
\end{align}
Now, we define
\begin{align}\label{ff}
\mathcal{J}_k^{+}
=
\Bigl\{
J_k:\;& |J_k|=\tau,
F_{I_k}(x_k)^{T}
\nabla F_{I_k,J_k}(x_k)
I_{:,J_k}^{T}
(x_k-\widehat{x})
>0
\Bigr\},
\end{align}
and
\begin{align}\label{gg}
\mathcal{J}_k^{-}
=
\Bigl\{
J_k:\;& |J_k|=\tau,
F_{I_k}(x_k)^{T}
\nabla F_{I_k,J_k}(x_k)
I_{:,J_k}^{T}
(x_k-\widehat{x})
<0
\Bigr\}.
\end{align}
Using
$\frac{1}{\underline{\nu}^{\,2}}
=
\frac{1}{\overline{\nu}^{\,2}}
+
\left(
\frac{1}{\underline{\nu}^{\,2}}
-
\frac{1}{\overline{\nu}^{\,2}}
\right),$ and combining \eqref{ff} and
\eqref{gg} yields
\begin{align}
&
-
\frac{1}{\binom{n}{\tau}}
\sum_{|J_k|=\tau}
\frac{
\left\|F_{I_k}(x_k)\right\|_2^{2}
}{
\left\|
\nabla F_{I_k,J_k}^{T}(x_k)
F_{I_k}(x_k)
\right\|_2^{2}
}
\nonumber
F_{I_k}(x_k)^{T}
\nabla F_{I_k,J_k}(x_k)
I_{:,J_k}^{T}
(x_k-\widehat{x})
\nonumber\\
\leq&
-
\frac{1}
{\overline{\nu}^{\,2}\binom{n}{\tau}}
\sum_{J_k\in\mathcal{J}_k^{+}}
F_{I_k}(x_k)^{T}
\nabla F_{I_k,J_k}(x_k)
I_{:,J_k}^{T}
(x_k-\widehat{x})
-
\frac{1}
{\underline{\nu}^{\,2}\binom{n}{\tau}}
\sum_{J_k\in\mathcal{J}_k^{-}}
F_{I_k}(x_k)^{T}
\nabla F_{I_k,J_k}(x_k)
I_{:,J_k}^{T}
(x_k-\widehat{x})\nonumber\\
&=-
\frac{1}
{\overline{\nu}^{\,2}\binom{n}{\tau}}
\sum_{|J_k|=\tau}
F_{I_k}(x_k)^{T}
\nabla F_{I_k,J_k}(x_k)
I_{:,J_k}^{T}
(x_k-\widehat{x})\nonumber\\
&+
\frac{1}{\binom{n}{\tau}}
\left(
\frac{1}{\underline{\nu}^{\,2}}
-
\frac{1}{\overline{\nu}^{\,2}}
\right)
\sum_{J_k\in\mathcal{J}_k^{-}}
\left|
F_{I_k}(x_k)^{T}
\nabla F_{I_k,J_k}(x_k)
I_{:,J_k}^{T}
(x_k-\widehat{x})
\right|.
\label{eq15}
\end{align}
Since
$\sum_{|J_k|=\tau}
I_{:,J_k}I_{:,J_k}^{T}
=\binom{n-1}{\tau-1}I_n$,
 and combing the inequality \eqref{eq15}, we can get
\begin{align}
\frac{1}{\binom{n}{\tau}}
\sum_{|J_k|=\tau}
F_{I_k}(x_k)^{T}
\nabla F_{I_k,J_k}(x_k)
I_{:,J_k}^{T}
(x_k-\widehat{x})
&\nonumber
=
\frac{\binom{n-1}{\tau-1}}
{\binom{n}{\tau}}
F_{I_k}(x_k)^{T}
\nabla F_{I_k,:}(x_k)
(x_k-\widehat{x})
\nonumber\\
&=
\frac{\tau}{n}
F_{I_k}(x_k)^{T}
\nabla F_{I_k,:}(x_k)
(x_k-\widehat{x}).
\label{eq:block-average-identity}
\end{align}
Since $F_{I_k}(\widehat{x})=0$, the tangential cone condition
(\ref{def422}) implies
\begin{equation}
\left\|
F_{I_k}(x_k)
-
\nabla F_{I_k,:}(x_k)
(x_k-\widehat{x})
\right\|_2
\leq
\eta
\left\|
F_{I_k}(x_k)
\right\|_2.
\end{equation}
Consequently,
\begin{align}
F_{I_k}(x_k)^{T}
\nabla F_{I_k,:}(x_k)
(x_k-\widehat{x})
\nonumber
&=
\left\|
F_{I_k}(x_k)
\right\|_2^{2}
\nonumber-
F_{I_k}(x_k)^{T}
\Bigl[
F_{I_k}(x_k)
-
\nabla F_{I_k,:}(x_k)
(x_k-\widehat{x})
\Bigr]
\nonumber\\
\geq&
\left\|
F_{I_k}(x_k)
\right\|_2^{2}
\nonumber-
\left\|
F_{I_k}(x_k)
\right\|_2
\left\|
F_{I_k}(x_k)
-
\nabla F_{I_k,:}(x_k)
(x_k-\widehat{x})
\right\|_2
\nonumber\\
\geq&
(1-\eta)
\left\|
F_{I_k}(x_k)
\right\|_2^{2}.
\label{eq:principal-inner-product-lower-bound}
\end{align}
Furthermore, by the Cauchy--Schwarz inequality, we obtain
\begin{align}
&
\frac{1}{\binom{n}{\tau}}
\sum_{J_k\in\mathcal{J}_k^{-}}
\left|
F_{I_k}(x_k)^{T}
\nabla F_{I_k,J_k}(x_k)
I_{:,J_k}^{T}(x_k-\widehat{x})
\right|
\nonumber
\leq
\frac{1}{\binom{n}{\tau}}
\sum_{|J_k|=\tau}
\left\|
\nabla F_{I_k,J_k}^{T}(x_k)
F_{I_k}(x_k)
\right\|_2
\left\|
I_{:,J_k}^{T}(x_k-\widehat{x})
\right\|_2
\nonumber\\
\leq{}&
\left[
\frac{1}{\binom{n}{\tau}}
\sum_{|J_k|=\tau}
\left\|
\nabla F_{I_k,J_k}^{T}(x_k)
F_{I_k}(x_k)
\right\|_2^{2}
\right]^{1/2}
\times
\left[
\frac{1}{\binom{n}{\tau}}
\sum_{|J_k|=\tau}
\left\|
I_{:,J_k}^{T}(x_k-\widehat{x})
\right\|_2^{2}
\right]^{1/2}\nonumber\\
&=\frac{\tau}{n}
\left\|
\nabla F_{I_k,:}^{T}(x_k)
F_{I_k}(x_k)
\right\|_2
\left\|
x_k-\widehat{x}
\right\|_2\leq \frac{\tau}{n}
\overline{\nu}
\left\|
F_{I_k}(x_k)
\right\|_2
\left\|
x_k-\widehat{x}
\right\|_2.
\label{eq22}
\end{align}
By the definition of the greedy row index set $I_k$ in
Algorithm \ref{alg3}, we obtain
\begin{equation}
\left\|
F_{I_k}(x_k)
\right\|_2^{2}
=
\sum_{i\in I_k}
\left|
F^{(i)}(x_k)
\right|^{2}
\geq
\theta
\max_{1\leq i\leq m}
\left|
F^{(i)}(x_k)
\right|^{2}
\geq
\frac{\theta}{m}
\left\|
F(x_k)
\right\|_2^{2}.
\label{eq:greedy-residual-bound}
\end{equation}
Applying the inequality (\ref{eq:block ltcc}), (\ref{eq:greedy-residual-bound}),  $F(\widehat{x})=0$
and $\nabla F(\widehat{x})$ has full column rank, we have
\begin{equation}
\label{eq19}
\left\|
F_{I_k}(x_k)
\right\|_2^{2}
\geq
\frac{
\theta
\sigma_{\min}^{2}
\bigl(
\nabla F(\widehat{x})
\bigr)
}{
m(1+\eta)^2
}
\left\|
x_k-\widehat{x}
\right\|_2^{2}.
\end{equation}
Substitution of \eqref{eq19} into
\eqref{eq22} yields
\begin{align}\label{eqqq}
&
\frac{1}{\binom{n}{\tau}}
\sum_{J_k\in\mathcal{J}_k^{-}}
\left|
F_{I_k}(x_k)^{T}
\nabla F_{I_k,J_k}(x_k)
I_{:,J_k}^{T}(x_k-\widehat{x})
\right|
\leq
\frac{\tau}{n}
\frac{
\overline{\nu}(1+\eta)\sqrt{m}
}{
\sqrt{\theta}\,
\sigma_{\min}
\bigl(
\nabla F(\widehat{x})
\bigr)
}
\left\|
F_{I_k}(x_k)
\right\|_2^{2}.
\end{align}
Under conditions \eqref{23} and
\eqref{24},
combining
\eqref{eq:principal-inner-product-lower-bound},
\eqref{eqqq}
,\eqref{eq:conditional-expectation} and \eqref{eq19}, we can get 
\begin{align}
\mathbb{E}_k
\left[
D_{\varphi}^{x_{k+1}^{*}}
(x_{k+1},\widehat{x})
\right]
\leq
D_{\varphi}^{x_k^{*}}
(x_k,\widehat{x})
-
\gamma\delta
\Bigg[
\frac{\tau}{n}
\frac{1-\eta}{\overline{\nu}^{\,2}}
-
\frac{\tau}{n}
\left(
\frac{1}{\underline{\nu}^{\,2}}
-
\frac{1}{\overline{\nu}^{\,2}}
\right)
\frac{
\overline{\nu}(1+\eta)\sqrt{m}
}{
\sqrt{\theta}\,
\sigma_{\min}
\bigl(
\nabla F(\widehat{x})
\bigr)
}
-
\frac{\delta}{2\underline{\nu}^{\,2}}
\Bigg]
\left\|
F_{I_k}(x_k)
\right\|_2^{2}
\nonumber\\
\leq
D_{\varphi}^{x_k^{*}}(x_k,\widehat{x})
-
\frac{
\gamma\delta\theta
\sigma_{\min}^{2}
\bigl(
\nabla F(\widehat{x})
\bigr)
}{
m(1+\eta)^2
}
\Bigg[
\frac{\tau}{n}\frac{1-\eta}{\overline{\nu}^{\,2}}
-
\frac{\tau}{n}
\left(
\frac{1}{\underline{\nu}^{\,2}}
-
\frac{1}{\overline{\nu}^{\,2}}
\right)
\frac{\overline{\nu}(1+\eta)\sqrt{m}}{\sqrt{\theta}\,\sigma_{\min}\bigl(\nabla F(\widehat{x})\bigr)}
-
\frac{\delta}{2\underline{\nu}^{\,2}}
\Bigg]
\left\|
x_k-\widehat{x}
\right\|_2^{2}.
\label{eq20}
\end{align}
Since $\varphi$ is $M$-smooth \rm\cite{Niu25a}, one has
$D_{\varphi}^{x_k^{*}}
(x_k,\widehat{x})
\leq
M
\left\|
x_k-\widehat{x}
\right\|_2^{2}.$
Thus, the inequality \eqref{eq20} can be rewritten as
\begin{equation}
\label{eq:conditional-contraction}
\mathbb{E}_k
\left[
D_{\varphi}^{x_{k+1}^{*}}
(x_{k+1},\widehat{x})
\right]
\leq
(1-c)
D_{\varphi}^{x_k^{*}}
(x_k,\widehat{x}),
\end{equation}
where $c$ is given by \eqref{256}.
Taking the full expectation on both sides of
\eqref{eq:conditional-contraction} and taking the tower property of
conditional expectation, we obtain
$\mathbb{E}
\left[
D_{\varphi}^{x_{k+1}^{*}}
(x_{k+1},\widehat{x})
\right]
\leq
(1-c)
\mathbb{E}
\left[
D_{\varphi}^{x_k^{*}}
(x_k,\widehat{x})
\right]
\nonumber\leq
(1-c)^{k+1}
D_{\varphi}^{x_0^{*}}
(x_0,\widehat{x}).$
Hence, we have $\mathbb{E}
\left[
D_{\varphi}^{x_k^{*}}
(x_k,\widehat{x})
\right]
\leq
(1-c)^k
D_{\varphi}^{x_0^{*}}
(x_0,\widehat{x}),$
which proves \eqref{eq:linear-bregman-convergence}.
Finally, by the $\gamma$-strong convexity of $\varphi$, we can get $ \frac{\gamma}{2}
\left\|
x_k-\widehat{x}
\right\|_2^{2}
\leq
D_{\varphi}^{x_k^{*}}
(x_k,\widehat{x}),$ which means that  the inequality (\ref{eq111}) is proved.
\end{proof}

\section{Numerical experiments}\label{sec4}
In this section, we present some numerical experiments to verify the  effectiveness of  the GDBNK method  in solving  sparse solutions of nonlinear equations. We compare the GDBNK method with the  NBK, rNBK, and ABNBK-a methods on various nonlinear problems. All experiments were performed with MATLAB (Version R2024a). The computer used
is equipped with a Windows 11 operating system, 16G memory, and a 2.60GHz central
processing unit (13th Gen Intel(R) Core(TM) i7-13650HX).

The index $i$ in the NBK and rNBK methods is selected using uniform probability sampling $\boldsymbol{\mathrm{ones}}(m,1)/m$. For the parameters of different methods, we set $\sigma=1$, $\gamma=1$, $\delta_{ABNBK-a}=1.3$, $\theta=0.1$, $\delta_{GDBNK}=1.2$.  The sparsity level of the sparse solution $\hat{x}$ is denoted by '$sp$'.   The sparsity level $sp$ is 0.05.   $\tau$
 denotes the number of columns selected from the Jacobian matrix 
$\nabla F$
 for updating.
We use $\varphi(x)=\lambda\|x\|_1+\frac{1}{2}\|x\|_2^2$ and set $\lambda=2$. The initial subgradient $x_0^*$ is generated from the standard normal distribution. Then the initial vector $x_0=\nabla\varphi^*(x_0^*)=S_\lambda(x_0^*)$, where \(S_\lambda(x_0^*)\) denotes the soft-thresholding operator.  We use the number of  iteration steps (IT) and  the computing time in seconds (CPU) to measure the effectiveness of these methods. The IT and CPU values are averaged over 10 independent trials. The stopping condition is defined as $\frac{||F(x_k)||_2^2}{||F(x_0)||_2^2}\leq 10^{-6}$ for the relative residual (RES), or when the iteration number reaches 10000.

\begin{example}\label{num1}
We consider solving the sparse solution of quadratic equations
\begin{equation*}
	F_i(x)=\frac{1}{2}\langle x,A^{(i)}x\rangle+\langle b^{(i)},x\rangle+c^{(i)}=0, \quad 1\leq i\leq m
\end{equation*}
with $A^{(i)}\in\mathbb{R}^{n\times n}$, $b^{(i)}\in\mathbb{R}^n$, $c^{(i)}\in\mathbb{R}$. We set 
$	c^{(i)}=-\left(\frac{1}{2}\langle x,A^{(i)}x\rangle+\langle 
b^{(i)},x\rangle\right)$,  $b^{(i)}$ and $\hat{x}$ can be generated by MATLAB function $rand(n,1)$ and $sprandn(n,1,sp)$, respectively. 
 $A^{(i)}$ is a random partial discrete cosine transform (DCT) matrix with j-th column generated through the expression as:
		$A^{(i)}_{:,j}=\cos(2\pi(j-1)\omega), $
	where $\omega\in\mathbb{R}^{n\times1}$ is a column vector with elements sampled uniformly and independently from $[0,1]$.
\end{example}

From Table \ref{tb2} and Figure \ref{fig1}, we observe that the GDBNK method consumes fewer iterations and computation time than the other algorithms, showing its competitive performance.


\begin{figure}
\centering
\begin{subfigure}{0.35\textwidth}
\centering
\includegraphics[width=\linewidth]{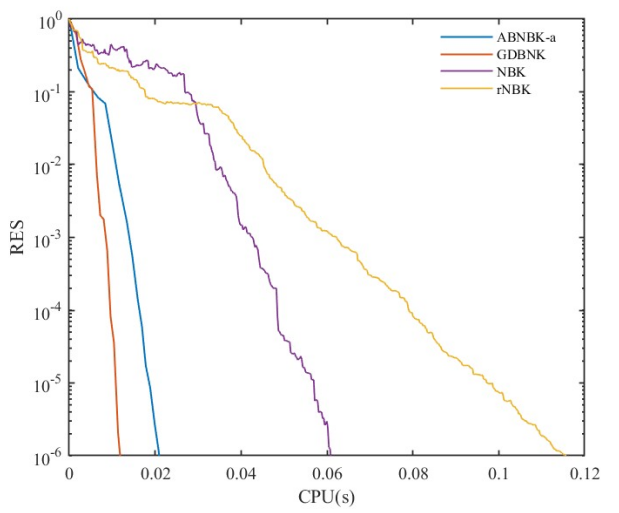}
\caption{$m=100, n=50, \tau=36$}
\label{fig:sub1}
\end{subfigure}
\begin{subfigure}{0.35\textwidth}
\centering
\includegraphics[width=\linewidth]{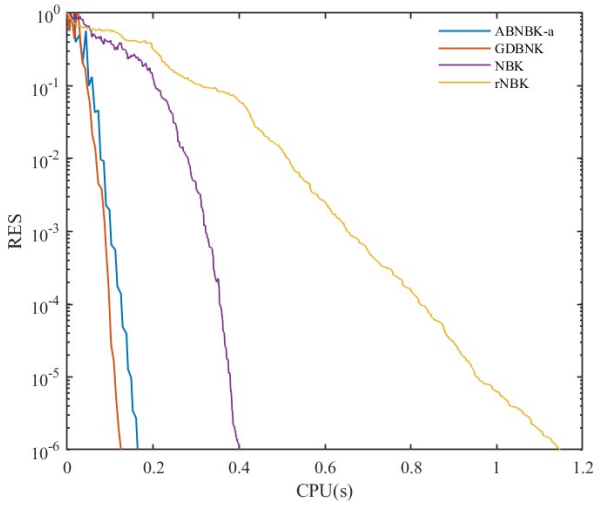}
\caption{$m=200, n=100, \tau=86$}
\label{fig:sub2}
\end{subfigure}
\caption{RES vs IT in Example \ref{num1}}
\label{fig1}
\end{figure}


\begin{table}
\centering
\caption{The average of IT and CPU of the four methods in Example \ref{num1} }\label{tb2}
 \footnotesize
\begin{tabular*}{\tblwidth}{@{}L L L  L L L L L L@{}}
\toprule
m & n  &$\tau$ & & ABNBK-a & GDBNK & NBK & rNBK \\
\midrule
100&50&36&IT&15&\textbf{13}&246&772 \\ 
&&&CPU&0.0218&\textbf{0.0120}&0.0610&0.1157 \\  
200&100&86&IT&25&\textbf{21}&503&1751 \\
&&&CPU&0.1642&\textbf{0.1261}&0.4002&1.1467 \\ 
300&150&135&IT&83&\textbf{59}&1366&5095  \\
&&&CPU&1.2513&\textbf{0.8804}&3.2751&11.4126\\ 
400&200&178&IT&41&\textbf{33}&1289&6834\\
&&&CPU&1.2331&\textbf{0.8756}&6.7922&35.0517\\ 
600&300&278&IT&38&\textbf{30}&3664&8094\\
&&&CPU&2.7148&\textbf{2.2444}&56.3121&125.5761\\ 
1000&500&450&IT&38&\textbf{31}&4251&>10000\\
&&&CPU&7.8956&\textbf{6.4268}&272.0773&>642.9233\\ 
\bottomrule
\end{tabular*}
\end{table}



\begin{figure} 
	\centering
\includegraphics[width=6.0in]{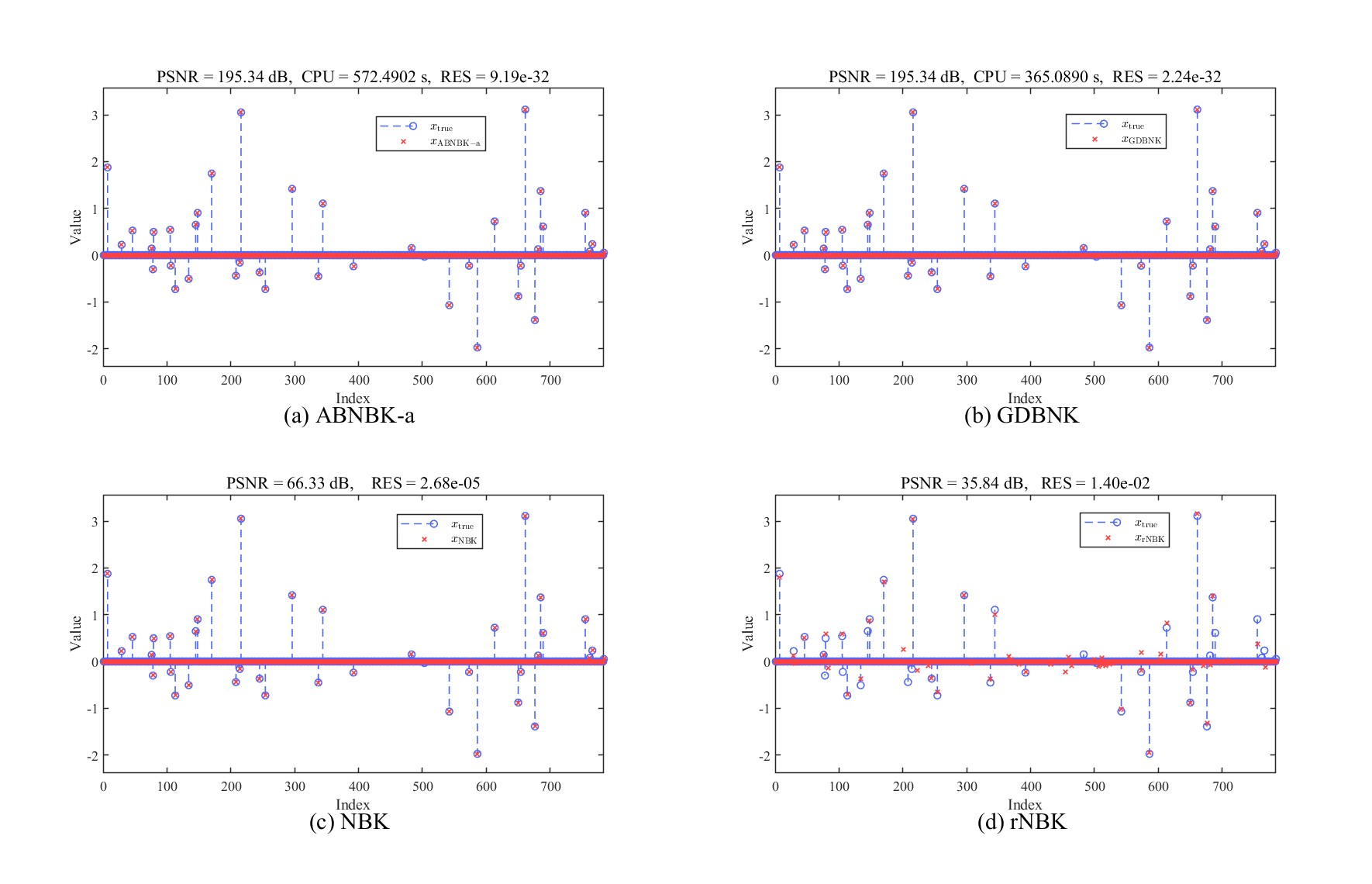}
	\caption[d]{ The original signal and recovered signal of four methods for $A^{(i)} $ sampled from the standard normal distribution  in Example \ref{333} with $m=2000, n=28^2, \tau=740$
	}\label{fig5}
\end{figure}



\begin{example}\label{333}
    We consider recovering the MNIST image recovery problem. The original signal is a $28\times28$ image $X_t$ in the MNIST dataset. Let the coefficient matrix $A$ be generated by the MATLAB function rand(m,n) with $m=2000$ and $n=28^2$. We reshape the original image matrix to get the vector $\hat{x}$. The right-hand side vector is $b=A\hat{x}+e$, where $e$ is a Gaussian noise vector with a noise level of 0.01. We use the peak signal-to-noise ratio (PSNR) to measure the quality of the recovered image, which is defined as
    \begin{equation*}
        PSNK(X_t,X_r)=10*log_{10}\frac{N^2max(X_t(i,j))^2}{\sum_{i=1}^N\sum_{j=1}^N|X_t(i,j)-X_r(i,j)|^2},
    \end{equation*}
    where $X_r$ is the recovered image. The initial subgradient $x_0^*$ is set to be a zero vector. 
\end{example}

Figure \ref{fig5} illustrates the results obtained after 3000 iterations of the NBK, rNBK, ABNBK-a, and GDBNK methods. 
The results show that the ABNBK-a and GDBNK methods achieve identical PSNR and residual level after 3000 iterations, both at 195.34 dB, which is considerably higher than the PSNR values of the NBK and rNBK methods. However, the GDBNK method requires significantly less CPU computation time than the ABNBK-a method.

\section{Conclusions}\label{sec5}
In this paper, we develop the GDBNK method for finding sparse solutions to nonlinear equations. Under appropriate assumptions, we analyze the convergence properties of the proposed method. Numerical experiments show that the GDBNK method outperforms  NBK, rNBK, and ABNBK-a methods in terms of IT and CPU time, which verifies its efficiency. Although the GDBNK algorithm yields a notable acceleration, it is still undetermined how to determine the range of the $\tau$ value
 and set associated parameters to achieve optimality, which will be our future research direction.
\section*{Declaration of Generative AI and AI-assisted technologies in the writing process}
We have not used Artificial Intelligence (AI) tools in the creation of this article.

\section*{Funding}
This work was supported by the Shandong Provincial Natural Science Foundation [grant numbers ZR2026MS0051] and the Fundamental Research Funds for the Central Universities of China [grant number 26CX03017A].

\section*{Data availability}
No data was used for the research described in the article.
\bibliographystyle{unsrt}
\bibliography{cas-refs}
\end{document}